\documentclass[reqno,11pt]{article}
\usepackage{a4wide,xcolor,eucal,enumerate,mathrsfs}
\usepackage[normalem]{ulem}

\usepackage[pdfborder={0 0 0}]{hyperref}  

\usepackage{amsmath,amssymb,epsfig,amsthm}
\usepackage[latin1]{inputenc}

\numberwithin{equation}{section}

\newtheorem{theorem}{Theorem}[section]

\newtheorem{lemma}[theorem]{Lemma}
\newtheorem{proposition}[theorem]{Proposition}

\newtheorem{definition}[theorem]{Definition}

\newcommand{\ppi}{{\mbox{\boldmath$\pi$}}}
\newcommand{\ggamma}{{\mbox{\boldmath$\gamma$}}}
\newcommand{\sggamma}{{\mbox{\scriptsize\boldmath$\gamma$}}}
\newcommand{\diam}{\mathop{\rm diam}\nolimits}
\newcommand{\supp}{\mathop{\rm supp}\nolimits}
\renewcommand{\d}{{\mathrm d}}
\newcommand{\N}{\mathbb{N}}
\newcommand{\R}{\mathbb{R}}
\newcommand{\LL}{\mathscr{L}}
\newcommand{\MM}{\mathscr{M}}
\newcommand{\mm}{\mathfrak m}
\newcommand{\nn}{\mathfrak n}
\newcommand{\sfd}{{\sf d}}
\newcommand{\prob}[1]{\mathscr P(#1)}
\newcommand{\probt}[1]{\mathscr P_2(#1)}
\newcommand{\eps}{\varepsilon}
\newcommand{\geo}{{\rm{Geo}}}
\newcommand{\e}{{\rm{e}}}
\newcommand{\gopt}{{\rm{OptGeo}}}
\newcommand{\Tan}{{\rm Tan}}
\newcommand{\lims}{\varlimsup}
\newcommand{\limi}{\varliminf}
\newcommand{\res}{\mathop{\hbox{\vrule height 7pt width .5pt depth 0pt
\vrule height .5pt width 6pt depth 0pt}}\nolimits}
\newcommand{\CD}{{\sf CD}}
\newcommand{\RCD}{{\sf RCD}}

\title{Euclidean spaces as weak tangents of infinitesimally Hilbertian metric spaces with Ricci curvature bounded below}
\begin{document}

\author{Nicola Gigli\
   \thanks{University of Nice, \textsf{nicola.gigli@unice.fr}}
   \and
   Andrea Mondino
   \thanks{ETH, Zurich, \textsf{andrea.mondino@math.ethz.ch}}
   \and
   Tapio Rajala
   \thanks{University of Jyv\"askyl\"a, \textsf{tapio.m.rajala@jyu.fi}}
   }

\maketitle

\begin{abstract}
We show that in any infinitesimally Hilbertian $\CD^*(K,N)$-space at almost every point there exists a Euclidean weak tangent,
i.e. there exists a sequence of dilations of the space that converges to a Euclidean space in the pointed measured Gromov-Hausdorff topology.
The proof follows by considering iterated tangents and the splitting theorem for infinitesimally Hilbertian $\CD^*(0,N)$-spaces.
\end{abstract}


\


\tableofcontents

\section{Introduction}

Gromov-Hausdorff limits of Riemannian manifolds with Ricci curvature lower bounds, Ricci-limits for short, have been extensively studied in particular by Cheeger and Colding
in a series of papers
\cite{Cheeger-Colding96, Cheeger-Colding97I, Cheeger-Colding97II, Cheeger-Colding97III, Colding96a, Colding96b, Colding97}.
In \cite{Cheeger-Colding97I} they proved, among other things, that the tangent space at almost every point - intended as pointed Gromov-Hausdorff limit of rescaled spaces - of a Ricci-limit space is Euclidean, with dimension possibly
depending on the point. 
Only much later, in \cite{Colding-Naber12} Colding-Naber showed that in fact for any Ricci-limit space there exists $k \in \N$ 
such that at almost every point in the space the tangent cone is $\R^k$.
Notice however, that there can be points in the Ricci-limits where the tangent is not unique, see for instance \cite{Colding-Naber13} for  examples.

In \cite{Lott-Villani09}  and  \cite{Sturm06I, Sturm06II} Lott-Villani on one side and Sturm on the other independently pro- posed a definition of `having Ricci curvature bounded from below by $K$ and dimension bounded above by $N$' for metric measure spaces, these being called $\CD(K,N)$-spaces (in \cite{Lott-Villani09} only the cases $K = 0$ or $N = \infty$ were considered). Here $K$ is a real number and $N$ a real number at least one, the value $N = \infty$ being also allowed.

The crucial properties of their definition are the compatibility with the smooth Riemannian case and the stability w.r.t. measured Gromov-Hausdorff convergence. 

More recently, in \cite{BS2010} Bacher-Sturm proposed a variant of the curvature-dimension condition $\CD(K,N)$,
 called reduced curvature-dimension condition and denoted as $\CD^*(K,N)$ which, while retaining the aforementioned
 stability and compatibility, has better globalization and tensorialization properties. For the special case $K=0$
 we have $\CD(0,N)=\CD^*(0,N)$. For arbitrary $K$ we have that any $\CD(K,N)$-space is also $\CD^*(K,N)$, but
 the converse implication is currently not perfectly understood (it is known that $\CD^*(K,N)$ implies $\CD(K^*,N)$,
 where $K^*=\frac{N-1}{N} K \leq K$ but the equivalence is open; for some recent progress see
 \cite{Cavalletti-Sturm12} and \cite{Cavalletti12}).

From both the geometric and analytic perspective, a delicate issue concerning the $\CD(K,N)$ and $\CD^*(K,N)$ conditions is that they comprehend Finsler structures (see the last theorem in \cite{Villani09}), which after the works of Cheeger-Colding are known not to appear as Ricci-limit spaces.

To overcome this problem, in \cite{Ambrosio-Gigli-Savare11b} the first author  together with Ambrosio-Savar\'e introduced, specifically for the case $N=\infty$, a more restrictive condition which retains the stability properties w.r.t. measured Gromov-Hausdorff convergence and rules out Finsler geometries. This notion is called Riemannian curvature bound and denoted by $\RCD(K,\infty)$. According to the slightly finer axiomatization presented in \cite{AmbrosioGigliMondinoRajala}, by the authors and Ambrosio it can be presented as the reinforcement of the $\CD(K,\infty)$ condition with the requirement that the space is `infinitesimally Hilbertian' (see also \cite{Gigli12}), the latter meaning that the Sobolev space $W^{1,2}(X,\sfd,\mm)$ of real valued functions on $(X,\sfd,\mm)$ is Hilbert (in general it is only Banach).

In \cite{Ambrosio-Gigli-Savare11b}, \cite{AGSBaEm} (see also \cite{GigliKuwadaOhta10} for the first progresses in this direction) it has been shown that the $\RCD(K,\infty)$ condition is equivalent to the (properly written/understood) Bochner inequality
\[
\Delta\frac{|\nabla f|^2}{2}\geq \nabla f\cdot\nabla \Delta f+K|\nabla f|^2.
\] 
The non-trivial refinement of this result to the finite dimensional case has been carried out in  \cite{Erbar-Kuwada-Sturm13} and \cite{AMS} by Erbar-Kuwada-Sturm and  the second author together with Ambrosio-Savar\'e respectively, where it has been proved that `infinitesimal Hilbertianity plus $\CD^*(K,N)$' is equivalent to the Bochner inequality
\[
\Delta\frac{|\nabla f|^2}{2}\geq\frac{(\Delta f)^2}{N}+\nabla f\cdot\nabla \Delta f+K|\nabla f|^2,
\] 
(again, properly understood).

Although  infinitesimal Hilbertianity is a requirement  analytic in nature, it has been shown 
in \cite{Gigli13} that on infinitesimally Hilbertian $\CD^*(0,N)$-spaces the analog of the 
Cheeger-Colding-Gromoll splitting theorem holds (see also \cite{GigliMosconi}
for the Abresch-Gromoll inequality), thus providing a geometric property which fails on general $\CD(K,N)/\CD^*(K,N)$-spaces.
Unlike general $\CD(K,N)/\CD^*(K,N)$-spaces, infinitesimally Hilbertian $\CD^*(K,N)$-spaces are also known to be essentially non branching \cite{RajalaSturm}.

Still in the direction of understanding the geometry of infinitesimally Hilbertian $\CD^*(K,N)$-spaces,  a natural conjecture is that on such setting  the tangent spaces (i.e. pointed measured Gromov-Hausdorff limits of rescaled spaces)  are Euclidean at almost every point. Moreover, like for the Ricci-limits, the tangents
should be unique at almost every point. Here we make a step towards these 
conjectures by proving the following result:

\begin{theorem}\label{thm:euclideantangents}
 Let $K \in \R$, $1 \le N < \infty$ and  $(X,\sfd,\mm)$  an infinitesimally Hilbertian $\CD^*(K,N)$-space.
 Then at $\mm$-almost every $x \in X$ there exists $n \in \N$, $n \le N$, such that 
 \[
  (\R^n,\sfd_{E},\LL_n,0) \in \Tan(X,\sfd,\mm,x),  
 \]
 where $\sfd_{E}$ is the Euclidean distance and $\LL_n$ is the $n$-dimensional
 Lebesgue measure normalized so that $\int_{B_1(0)}1-|x|\,\d\LL_n(x)=1$.
\end{theorem}
Here $\Tan(X,\sfd,\mm,x)$ denotes the collection of pointed measured Gromov-Hausdorff 
limits of rescaled spaces centered at $x$. Notice that the normalization of the limit measure expressed in the statement plays little role  and depends only on the choice of renormalization of rescaled measures in the process of taking limits.

The idea for the proof of Theorem  \ref{thm:euclideantangents} is the one used by Cheeger-Colding in \cite{Cheeger-Colding97I}, namely to prove that $\mm$-a.e. point is the middle point of a non-constant geodesic, noticing that in the limit of blow-ups the space becomes an infinitesimally Hilbertian $\CD^*(0,N)$-space and the  geodesics a line, then to use the splitting to factorize a direction. At this point it is a matter of proving that one can factorize enough dimensions to deduce that the limit is really Euclidean. In order to do so, Cheeger-Colding used some additional geometric information that is currently unavailable in the non-smooth setting: this is why we can't really prove that every tangent is Euclidean but only the existence of such tangent space.

Instead, we use a crucial idea of Preiss \cite{P1987}, adapted by Le Donne \cite{LD2011} to the metric-measure setting, which states that on doubling metric-measure spaces `tangents of tangents are tangents themselves', see Theorem \ref{thm:iteratedtangents} for the precise statement. Notice that we report the proof of such result because Le Donne stated the theorem for pointed Gromov-Hausdorff convergence, while we need it for the pointed measured Gromov-Hausdorff topology. Yet, such variant presents no additional difficulties so that we will basically just follow Le Donne's argument keeping track of the measures involved.

Finally, we remark that given that Theorem \ref{thm:euclideantangents} is proved via such compactness argument, in fact we prove the following slightly stronger statement: for every sequence of scalings we have that for $\mm$-a.e. $x$ there exists a subsequence (possibly depending on $x$) converging to a Euclidean space.

\smallskip
\noindent {\bf Acknowledgment.}
A.M. acknowledges the support of the ETH fellowship, part of the work was written when he was supported by the ERC grant GeMeTheNES directed by Prof. Luigi Ambrosio.
T.R. acknowledges the support of the Academy of Finland project no. 137528.

\section{Preliminaries}\label{sec:preli}
\subsection{Pointed metric measure spaces}
The basic objects we will deal with throughout the paper are metric measure spaces and pointed metric measure spaces, m.m.s. and p.m.m.s. 
 for short. 
 
 For our purposes, a m.m.s. is a triple $(X,\sfd,\mm)$ where $(X,\sfd)$ is a complete and separable metric space and  $\mm$ is a boundedly finite (i.e. finite on bounded subset) non-negative complete Borel measure  on it.

We will mostly work under the assumption that the measure 
$\mm$ is \emph{boundedly doubling}, i.e. such that
\begin{equation}\label{eq:mmdoubl}
0< \mm(B_{2r}(x)) \leq C(R) \mm(B_r(x)),\qquad\forall x\in X,\ r\leq R,
\end{equation}
for some given constants $C(R)>0$ depending on $R>0$. Notice that the map $C:(0,\infty)\to(0,\infty)$ can, and will, be taken non-decreasing.

The bound  \eqref{eq:mmdoubl} implies that $\supp \mm=X$ and $\mm \neq 0$ and by  iteration one gets
\begin{equation}\label{eq:genDoub}
 {\mm(B_R(a))}\leq \mm(B_r(x))\big(C(R)\big)^{\log_2(\frac{r}{R})+2}, \qquad\forall 0<r\leq R,\ a\in X,\ x\in B_R(a).
\end{equation}
In particular, this shows that bounded subsets are totally bounded and hence that boundedly doubling spaces are proper. 

\smallskip

A p.m.m.s is a quadruple $(X,\sfd,\mm,\bar x)$ where $(X,\sfd,\mm)$ is a metric measure space and $\bar x\in \supp(\mm)$ is a given `reference' point. Two p.m.m.s.  $(X,\sfd,\mm, \bar{x})$,  $(X',\sfd',\mm', \bar{x}')$ are declared isomorphic provided there exists an isometry $T:(\supp(\mm),\sfd)\to(\supp(\mm'),\sfd')$ such that $T_\sharp\mm=\mm'$ and $T(\bar{x})=\bar{x}'$.

We say that a p.m.m.s. $(X,\sfd,\mm,\bar x)$ is normalized provided $\int_{B_1(\bar x)}1-\sfd(\cdot,\bar x)\,\d\mm=1$. Obviously, given any p.m.m.s. $(X,\sfd,\mm,\bar x)$ there exists a unique $c> 0$ such that $(X,\sfd,c\mm,\bar x)$ is normalized, namely $c:=(\int_{B_1(\bar x)}1-\sfd(\cdot,\bar x)\,\d\mm)^{-1}$.

We shall denote by $\MM_{C(\cdot)}$ the class of (isomorphism classes of) normalized p.m.m.s. fulfilling \eqref{eq:mmdoubl} for given non-decreasing $C:(0,\infty)\to(0,\infty)$.

\subsection{Pointed measured Gromov-Hausdorff topology and  measured tangents}

The definition of convergence of p.m.m.s. that we shall adopt is the following (see \cite{Villani09}, \cite{GigliMondinoSavare}
 and \cite{BuragoBuragoIvanov}):
\begin{definition}[Pointed measured Gromov-Hausdorff convergence]\label{def:conv}
A sequence $(X_i,\sfd_i,\mm_i,\bar{x}_i)$ is said to converge 
in the  pointed measured Gromov-Hausdorff topology (p-mGH for short) to 
$(X_\infty,\sfd_\infty,\mm_\infty,\bar{x}_\infty)$ if and only if there 
exists a separable metric space $(Z,\sfd_Z)$ and isometric embeddings  
$\{\iota_i:(\supp(\mm_i),\sfd_i)\to (Z,\sfd_Z)\}_{i \in \bar{\N}}$ such that
for every 
$\varepsilon>0$ and $R>0$ there exists $i_0$ such that for every $i>i_0$
\[
\iota_\infty(B^{X_\infty}_R(\bar{x}_\infty)) \subset B^Z_{\varepsilon}[\iota_i(B^{X_i}_R (\bar{x}_i))]  \qquad \text{and} \qquad  \iota_i(B^{X_i}_R(\bar{x}_i)) \subset B^Z_{\varepsilon}[\iota_\infty(B^{X_\infty}_R (\bar{x}_\infty))], 
\]
where $B^Z_\varepsilon[A]:=\{z \in Z: \, \sfd_Z(z,A)<\varepsilon\}$ for every subset $A \subset Z$, and 
\[
\int_Y \varphi \, \d ((\iota_i)_\sharp(\mm_i))\qquad    \to \qquad  \int_Y \varphi \, \d  ((\iota_\infty)_\sharp(\mm_\infty)) \qquad \forall \varphi \in C_b(Z), 
\]
where $C_b(Z)$ denotes the set of real valued bounded continuous functions with bounded support in $Z$.
\end{definition}
Sometimes in the following, for simplicity of notation, we will identify the spaces $X_i$ with 
their isomorphic copies $\iota_i(X_i)\subset Z$. 

It is obvious that this is in fact a notion of convergence for isomorphism classes of p.m.m.s., the following proposition also follows by standard means, see e.g. \cite{GigliMondinoSavare} for details:
\begin{proposition}\label{prop:comp}
Let $C:(0,\infty)\to (0,\infty)$ be a non-decreasing function. Then there exists a distance $\mathcal D_{C(\cdot)}$ on $\MM_{C(\cdot)}$ for which converging sequences are precisely those converging in the p-mGH sense. Furthermore, the space  $(\MM_{C(\cdot)},\mathcal D_{C(\cdot)})$  is compact.
\end{proposition}
Notice that the compactness of  $(\MM_{C(\cdot)},\mathcal D_{C(\cdot)})$ follows by the standard argument of Gromov: the measures of spaces in  $\MM_{C(\cdot)}$ are uniformly boundedly doubling, hence balls of given radius around the reference points are uniformly totally bounded and thus compact in the GH-topology. Then weak compactness of the measures follows using the doubling condition again and the fact that they are normalized.

\bigskip

The object of study of this paper are measured tangents, which are defined as follows. Let  $(X,\sfd,\mm)$ be a m.m.s.,  $\bar x\in \supp(\mm)$ and $r\in(0,1)$; we consider the rescaled p.m.m.s. $(X,r^{-1}\sfd,\mm^{\bar{x}}_r,\bar x)$ where the measure $\mm^{\bar x}_r$ is given by
\begin{equation}
\label{eq:normalization}
\mm^{\bar x}_r:=\left(\int_{B_r(\bar x)}1-\frac 1r\sfd(\cdot,\bar x)\,\d\mm\right)^{-1}\mm.
\end{equation}
Then we define:
\begin{definition}[The collection of tangent spaces $\Tan(X,\sfd,\mm,\bar{x})$]
Let  $(X,\sfd,\mm)$ be a m.m.s. and  $\bar x\in \supp(\mm)$. A p.m.m.s.  $(Y,\sfd_Y,\nn,y)$ is called a
\emph{tangent} to $(X,\sfd,\mm)$ at $\bar{x} \in X$ if there exists a sequence of radii $r_i \downarrow 0$ so that
$(X,r_i^{-1}\sfd,\mm^{\bar{x}}_{r_i},\bar{x}) \to (Y,\sfd_Y,\nn,y)$ as 
$i \to \infty$ in the pointed measured Gromov-Hausdorff topology.

We denote the collection of all the tangents of $(X,\sfd,\mm)$ at 
$\bar{x} \in X$ by $\Tan(X,\sfd,\mm,\bar{x})$. 
\end{definition}
Notice that if $(X,\sfd,\mm)$ satisfies \eqref{eq:mmdoubl} for some non-decreasing $C:(0,\infty)\to(0,\infty)$, then $(X,r^{-1}\sfd,\mm^x_r,\bar x)\in \MM_{C(\cdot)}$ for every $\bar x\in X$ and $r\in(0,1)$ and hence   the compactness stated in Proposition \ref{prop:comp} ensures that the set $\Tan(X,\sfd,\mm,\bar{x})$ is non-empty. 

It is also worth to notice that the map
\[
(\supp(\mm),\sfd)\ni x\qquad\mapsto \qquad (X,\sfd,\mm^x_r,x),
\]
is (sequentially) continuous for every $r>0$, the target space being endowed with the p-mGH convergence.

\subsection{Lower Ricci curvature bounds}
Here we quickly recall those basic definitions and properties of spaces with lower Ricci curvature bounds that we will need later on.

We denote by $\prob X$ the space of Borel probability measures on the complete and separable metric space $(X,\sfd)$ and by $\probt X \subset \prob X$ the subspace consisting of all the probability measures with finite second moment.

For $\mu_0,\mu_1 \in \probt X$ the quadratic transportation distance $W_2(\mu_0,\mu_1)$ is defined by
\begin{equation}\label{eq:Wdef}
  W_2^2(\mu_0,\mu_1) = \inf_\sggamma \int_X \sfd^2(x,y) \,\d\ggamma(x,y),
\end{equation}
where the infimum is taken over all $\ggamma \in \prob{X \times X}$ with $\mu_0$ and $\mu_1$ as the first and the second marginal.

Assuming the space $(X,\sfd)$ to be geodesic, also the space $(\probt X, W_2)$ is geodesic. We denote  by $\geo(X)$ the space of (constant speed minimizing) geodesics on $(X,\sfd)$ endowed with the $\sup$ distance, and by $\e_t:\geo(X)\to X$, $t\in[0,1]$, the evaluation maps defined by $\e_t(\gamma):=\gamma_t$. It turns out that any geodesic $(\mu_t) \in \geo(\probt X)$ can be lifted to a measure $\ppi \in \prob{\geo(X)}$, so that $(\e_t)_\#\ppi = \mu_t$ for all $t \in [0,1]$. Given $\mu_0,\mu_1\in\probt X$, we denote by $\gopt(\mu_0,\mu_1)$ the space of all
$\ppi \in \prob{\geo(X)}$ for which $(\e_0,\e_1)_\#\ppi$ realizes the minimum in \eqref{eq:Wdef}. If $(X,\sfd)$ is geodesic, then the set $\gopt(\mu_0,\mu_1)$ is non-empty for any $\mu_0,\mu_1\in\probt X$.

We turn to the formulation of the $\CD^*(K,N)$ condition, coming from  \cite{BS2010}, to which we also refer for a detailed discussion of its relation with the $\CD(K,N)$ condition
 (see also \cite{Cavalletti-Sturm12} and \cite{Cavalletti12}).

Given $K \in \R$ and $N \in [1, \infty)$, we define the distortion coefficient $[0,1]\times\R^+\ni (t,\theta)\mapsto \sigma^{(t)}_{K,N}(\theta)$ as
\[
\sigma^{(t)}_{K,N}(\theta):=\left\{
\begin{array}{ll}
+\infty,&\qquad\textrm{ if }K\theta^2\geq N\pi^2,\\
\frac{\sin(t\theta\sqrt{K/N})}{\sin(\theta\sqrt{K/N})}&\qquad\textrm{ if }0<K\theta^2 <N\pi^2,\\
t&\qquad\textrm{ if }K\theta^2=0,\\
\frac{\sinh(t\theta\sqrt{K/N})}{\sinh(\theta\sqrt{K/N})}&\qquad\textrm{ if }K\theta^2 <0.
\end{array}
\right.
\]
\begin{definition}[Curvature dimension bounds]
Let $K \in \R$ and $ N\in[1,  \infty)$. We say that a m.m.s.  $(X,\sfd,\mm)$
 is a $\CD^*(K,N)$-space if for any two measures $\mu_0, \mu_1 \in \prob X$ with support  bounded and contained in $\supp(\mm)$ there
exists a measure $\ppi \in \gopt(\mu_0,\mu_1)$ such that for every $t \in [0,1]$
and $N' \geq  N$ we have
\begin{equation}\label{eq:CD-def}
-\int\rho_t^{1-\frac1{N'}}\,\d\mm\leq - \int \sigma^{(1-t)}_{K,N'}(\sfd(\gamma_0,\gamma_1))\rho_0^{-\frac1{N'}}+\sigma^{(t)}_{K,N'}(\sfd(\gamma_0,\gamma_1))\rho_1^{-\frac1{N'}}\,\d\ppi(\gamma)
\end{equation}
where for any $t\in[0,1]$ we  have written $(\e_t)_\sharp\ppi=\rho_t\mm+\mu_t^s$  with $\mu_t^s \perp \mm$.
\end{definition}
Notice that if $(X,\sfd,\mm)$ is a $\CD^*(K,N)$-space, then so is $(\supp(\mm),\sfd,\mm)$, hence it is not restrictive to assume that $\supp(\mm)=X$, a hypothesis that we shall always implicitly do from now on. It is also immediate to establish that
\begin{equation}
\label{eq:cdinv}
\begin{split}
&\text{If }(X,\sfd,\mm)\text{ is $\CD^*(K,N)$, then the same is true for $(X,\sfd,c\mm)$ for any }c>0.\\
&\text{If }(X,\sfd,\mm)\text{ is $\CD^*(K,N)$, then for $\lambda>0$ the space $(X,\lambda \sfd,\mm)$ is $\CD^*(\lambda^{-2}K,N)$}.
\end{split}
\end{equation}

On $\CD^*(K,N)$ a natural version of the Bishop-Gromov volume growth estimate holds (see \cite{BS2010} for the precise statement), it follows that  for any given $K\in\R$, $N\in[1,\infty)$ there exists a function $C:(0,\infty)\to(0,\infty)$ depending on $K,N$ such that any $\CD^*(K,N)$-space $(X,\sfd,\mm)$ fulfills \eqref{eq:mmdoubl}. 

In order to enforce, in some weak sense, a Riemannian-like behavior of spaces with a curvature-dimension bound, in \cite{Ambrosio-Gigli-Savare11b} (see also \cite{AmbrosioGigliMondinoRajala}, \cite{Gigli12}, \cite{AMS}, \cite{Erbar-Kuwada-Sturm13}, \cite{MonAng}), a strengthening of the $\CD^*(K,N)$ has been proposed: it consists in requiring that the space $(X,\sfd,\mm)$ is such that the Sobolev space $W^{1,2}(X,\sfd,\mm)$ is Hilbert, a condition we shall refer to as `infinitesimal Hilbertianity'. It is out of the scope of this note to provide full details about the definition of $W^{1,2}(X,\sfd,\mm)$ and its relevance in connection with Ricci curvature lower bounds. We will instead be satisfied in recalling the two crucial properties which are relevant for our discussion: the stability (see \cite{GigliMondinoSavare} and references therein) and the splitting theorem (see \cite{Gigli13}):
\begin{theorem}[Stability]\label{thm:stab}
Let $K\in \R$ and $N\in[1,\infty)$. Then the class of normalized p.m.m.s $(X,\sfd,\mm,\bar x)$ such that $(X,\sfd,\mm)$ is infinitesimally Hilbertian and $\CD^*(K,N)$ is closed (hence compact) w.r.t. p-mGH convergence.
\end{theorem}
\begin{theorem}[Splitting]\label{thm:splitting}
 Let $(X,\sfd,\mm)$ be an infinitesimally Hilbertian $\CD^*(0,N)$-space with $1 \le N < \infty$. Suppose that
 $\supp(\mm)$ contains a line. Then $(X,\sfd,\mm)$ is isomorphic to $(X'\times \R, \sfd'\times \sfd_E,\mm'\times \LL_1)$,
where $\sfd_E$ is the Euclidean distance, $\LL_1$ the Lebesgue measure and  $(X',\sfd',\mm')$ is an infinitesimally Hilbertian $\CD^*(0,N-1)$-space if $N \ge 2$ and a singleton if $N < 2$.
\end{theorem}

Notice that for the particular case $K=0$ the $\CD^*(0,N)$ condition is the same as the $\CD(0,N)$ one. Also, in the statement of the splitting theorem, by `line' we intend an isometric embedding of $\R$.

Notice that Theorem \ref{thm:stab} and properties \eqref{eq:cdinv} ensure that for any $K,N$ we have that
\begin{equation}
\label{eq:tancd}
\begin{split}
&\text{If }(X,\sfd,\mm)\text{ is an infinitesimally Hilbertian $\CD^*(K,N)$-space and $x\in X$ we have}\\
&\text{that every $(Y,\sfd,\nn,y)\in\Tan(X,\sfd,\mm,x)$ is infinitesimally Hilbertian and $\CD^*(0,N)$.}
\end{split}
\end{equation}

\section{Proof of the main result}\label{sec:proof}

We will first show that at almost every point in a $\CD^*(K,N)$-space there exist a geodesic for which the point is an interior point.
After that, we prove that iterated tangents of $\CD^*(K,N)$-spaces are still tangents of the original space (actually we prove this part in the slightly more general framework of m.m.s. satisfying \eqref{eq:mmdoubl} ).
Finally,  we use the interior points of geodesics and iterated tangents together with the splitting theorem
(Theorem \ref{thm:splitting}) to conclude the proof of Theorem \ref{thm:euclideantangents}.

\subsection{Prolongability of geodesics}
The following result is a simple consequence of the definition of $\CD^*(K,N)$-space. The same argument was used in
\cite{Gigli12b} and \cite{R2011}, which were in turn inspired by some ideas in \cite{LV2007}. 
\begin{lemma}[Prolongability of geodesics]\label{lma:prolonging}
 Let $K \in \R$, $ N \in[1, \infty)$ and  $(X,\sfd,\mm)$ be  a $\CD^*(K,N)$-space that is not a singleton.
 Then at $\mm$-almost every $x \in X$ there exists a non-constant geodesic $\gamma \in \geo(X)$ so that $\gamma_\frac{1}{2} = x$.
\end{lemma}
\begin{proof}
Take $x_0 \in X$ and $R > 0$. 
Define $\mu_0 = \frac{1}{\mm(B_R(x_0))}\mm\res B_R(x_0)$ and $\mu_1 = \delta_{x_0}$.
Let $\ppi \in \gopt(\mu_0,\mu_1)$ be the measure satisfying \eqref{eq:CD-def}. With the notation of \eqref{eq:CD-def} we then have
\[
-\int\rho_t^{1-\frac1{N}}\,\d\mm\leq - \int \sigma^{(1-t)}_{K,N}(\sfd(\gamma_0,\gamma_1))\rho_0^{-\frac1{N}}\,\d\ppi(\gamma)\qquad\to\qquad - \mm(B_R(x_0))^{\frac1N},  \quad \textrm{as }t \downarrow 0.
\]
Let us write $E_t := \{x \in X \,:\, \rho_t(x)>0\}$. By Jensen's inequality we get 
 \[
 - \int \rho_t^{1-\frac1N}\,\d\mm =-\int_{E_t} \rho_t^{1-\frac1N}\,\d\mm 
                               \ge -\mm(E_t)\left(\frac1{\mm(E_t)}\int_{E_t} \rho_t \,\d\mm\right)^{1-\frac1N}
                                \ge - \mm(E_t)^{\frac1N}.
 \]
 Since the optimal transport is performed along geodesics connecting the points of $B_R(x_0)$ to $x_0$, we have the inclusion $E_t\subset B_R(x_0)$; therefore $\mm(E_t) \to \mm(B_R(x_0))$ as $t \downarrow 0$, hence $\mm$-a.e. $x\in B_R(x_0)$  belongs to $E_t$ for some $t_x>0$. By construction, for $\mm$-a.e.  $x \in E_t$ there exists $\gamma \in \geo(X)$ and $t>0$ with $\gamma_t = x$ and $\gamma_1 = x_0$, thus  for $\mm$-a.e. $x\in B_R(x_0)$ there exists a non-constant geodesic $\gamma \in \geo(X)$ so that $\gamma_\frac{1}{2} = x$.
 The conclusion follows by covering the space $X$ with countably many balls.
\end{proof}

\subsection{Tangents of tangents are tangents}
In this subsection we adapt the celebrated theorem of Preiss \cite{P1987} of iterated tangents of measures in $\R^n$ to our setting. In particular we are
inspired  by  \cite[Theorem 1.1]{LD2011},  where Le Donne proved that for metric
spaces with doubling measure almost everywhere the tangents of tangents are tangents of the original space.  
 The difference here is that we also  include the weak convergence of measures to the notion of tangents.

\begin{theorem}[`Tangents of tangents are tangents']\label{thm:iteratedtangents}
Let  $(X,\sfd,\mm)$ be a m.m.s. satisfying \eqref{eq:mmdoubl} for some $C:(0,\infty)\to(0,\infty)$.

Then for  $\mm$-a.e.  $x \in X$ the following holds:  for any $(Y,\sfd_Y,\nn,y) \in \Tan(X,\sfd,\mm,x)$ and any $y' \in Y$ we have
 \[
  \Tan(Y,\sfd_Y,\nn^{y'}_1,y') \subset \Tan(X,\sfd,\mm,x),
 \]
 the measure $\nn^{y'}_1$ being defined as in \eqref{eq:normalization}.
\end{theorem}

\begin{proof}
We need to prove that
\[
\begin{split}
 \mm\big(\big\{x \in X \,:\, &\text{there exist } (Y,\sfd_Y,\nn,y) \in \Tan(X,\sfd, \mm,x) \text{ and}\\
   & y' \in Y \text{ such that }(Y,\sfd_Y,\nn_{1}^{y'},y') \notin \Tan(X,\sfd,\mm,x) \big\}\big) = 0,
\end{split}
\]
This will follow if we can show that for all $k,m \in \N$ one has
\begin{align*}
 \mm\bigg(\bigg\{x \in X \,:\, &\text{there exist } (Y,\sfd_Y,\nn,y) \in \Tan(X,\sfd,\mm,x) \text{ and }y' \in Y \text{ such that }\sfd_Y(y,y')\leq m\\
       &\text{and } \mathcal D_{C(\cdot)}\left((Y,\sfd_Y,\nn^{y'}_{1},y'),(X,r^{-1}\sfd,\mm^x_r,x)\right) \geq2 k^{-1} \text{ for all }r \in (0,m^{-1})\bigg\}\bigg) = 0.
\end{align*}
Fix  $k,m \in \N$ and notice that since $(\MM_{C(\cdot)},\mathcal D_{C(\cdot)})$ is compact, it is also separable. Hence it is sufficient to show that for any closed set $\mathcal U \subset \MM_{C(\cdot)}$ with $\diam_{\mathcal D_{C(\cdot)}}(\mathcal U) < (2k)^{-1}$ the set
\begin{align*}
A =  \bigg\{x \in &X \,:\, \exists\ (Y,\sfd_Y,\nn,y) \in \Tan(X,\sfd,\mm,x) \text{ and }y' \in Y \text{ such that }(Y,\sfd_Y,\nn^{y'}_{1},y') \in \mathcal U,\\
       &\sfd_Y(y,y')\leq m \text{ and }\mathcal D_{C(\cdot)}\left((Y,\sfd_Y,\nn^{y'}_{1},y'),(X,r^{-1}\sfd,\mm^x_r,x)\right) \geq 2k^{-1} \ \forall r \in (0,m^{-1})\bigg\},
\end{align*}
has $\mm$-measure zero.  We start by proving that $A$ is Suslin, and thus $\mm$-measurable. To this aim, let $\mathcal A\subset X\times\MM_{C(\cdot)}$ be given by the couples $\big(x,(Y,\sfd_Y,\nn,y)\big)$ with $(Y,\sfd_Y,\nn,y)\in \Tan(X,\sfd,\mm,x)$ and recall that for every $r\in\R$ the map $X\ni x\mapsto(X,r\sfd,\mm^x_{1/r},x)\in (\MM_{C(\cdot)},\mathcal D_{C(\cdot)})$ is continuous. Thus the set $\cup_{x\in X}\{x\}\times B_{1/i}(X,r\sfd,\mm^x_{1/r},x)\subset X\times\MM_{C(\cdot)}$ is open and hence the identity
\[
\mathcal A=\bigcap_{i\in\N}\bigcap_{j\in\N}\bigcup_{r\geq j}\bigcup_{x\in X}\{x\}\times B_{1/i}(X,r\sfd,\mm^x_{1/r},x)
\]
shows that $\mathcal A\subset X\times\MM_{C(\cdot)}$ is Borel. Next notice that the set $\mathcal B\subset\MM_{C(\cdot)}$ defined by
\[
\begin{split}
\mathcal B:= \mathcal U\cap\Big\{(Y,\sfd_Y,\nn,\bar y) \ :\ \mathcal D_{C(\cdot)}\big((Y,\sfd_Y,\nn,\bar y),(X,r^{-1}\sfd,\mm^x_r,x)\big)\geq 2k^{-1},\ \forall r \in (0,m^{-1})\Big\},
\end{split}
\] 
is closed. Then using the fact that spaces in $\MM_{C(\cdot)}$ are proper it is easy to deduce that the set $\mathcal C\subset\MM_{C(\cdot)}$ given by
\[
\mathcal C:=\Big\{(Y,\sfd_Y,\nn,\bar y) \ :\ \exists y'\in Y\text{ such that }\sfd_Y(y',\bar y)\leq m\text{ and }(Y,\sfd_Y,\nn^{y'}_1,y')\in \mathcal B\Big\},
\]
is closed as well. Hence, as $A$ is the projection on the first factor of $\mathcal A\cap (X\times\mathcal C)$, it is Suslin, as claimed.

Now we proceed by contradiction and assume that for some  $k,m,\mathcal U$ and $A$ as above one has $\mm(A)>0$. Let $a \in A$ be an $\mm$-density point of $A$, i.e. 
\begin{equation}\label{eq:amDP}
\lim_{r\downarrow 0} \frac{\mm(A\cap B_r(a))}{\mm(B_r(a))}=1.
\end{equation} 
Since $a \in A$, there exist $(Y, \sfd_Y,\nn,y) \in \Tan(X,\sfd,\mm,a)$
and $y' \in Y$ such that $(Y,\sfd_Y,\nn^{y'}_{1},y') \in \mathcal U$ and the fact that  $(Y, \sfd_Y,\nn,y) \in \Tan(X,\sfd,\mm,a)$ grants the existence of a sequence $r_i \downarrow 0$ such that
\begin{equation}\label{eq:XriatoY}
  (X,r_i^{-1}\sfd,\mm^a_{r_i},a) \to (Y, \sfd_Y,\nn,y) \quad \text{p-mGH}.
\end{equation}
Let $(Z,\sfd_Z)$ be the separable metric space and $\iota_i$, $i\in \N\cup\{\infty\}$, the isometric embeddings given by the
definition of p-mGH convergence in \eqref{eq:XriatoY}.
It is then immediate to check directly from Definition \ref{def:conv} that there exists a sequence $\{x_i\}_{i \in \N}\subset X$  such that
\begin{equation}\label{eq:xitoy'}
\lim_{i\to \infty}\sfd_Z(\iota_i(x_i),\iota_\infty(y'))=0. 
\end{equation}
Notice that we have
\[
\begin{split}
\lims_{i\to\infty}\sfd_Z(\iota_i(x_i),\iota_i(a))&\leq\lims_{i\to\infty}\sfd_Z(\iota_i(x_i),\iota_\infty(y'))+\sfd_Z(\iota_\infty(y'),\iota_\infty(y))+\lims_{i\to\infty}\sfd_Z(\iota_\infty(y),\iota_i(a))\\
&=\sfd_Z(\iota_\infty(y'),\iota_\infty(y)),
\end{split}
\]
and thus by the definition of rescaled metrics we get
\begin{equation}\label{eq:d(a,xi)}
\sfd (a,x_i) \leq C' r_i,\qquad\forall i\in\N,
\end{equation} 
for some constant $C'>0$.

\textbf{Claim}: \emph{There exists a sequence $\{a_i\}_{i \in \N} \subset A \subset X$ such that} 
\begin{equation}\label{eq:claim}
 \lim_{i\to \infty} \sfd_Z(\iota_i(a_i),\iota_\infty(y')) = 0.
\end{equation}
\emph{Proof of the claim}. Here we use the fact that $\mm$ is locally doubling. This assumption is needed to deduce that for every $\varepsilon>0$ there exists $i_0 \in \N$
 such that for every $i\geq i_0$ we have  $A\cap B_{\varepsilon r_i} (x_i)\neq\emptyset$.
 Indeed if it is not the case, there exists  $\varepsilon_0>0$ such that it holds
\begin{equation}\label{eq:AcapB}
A  \cap  B_{\varepsilon_0 r_i} (x_i) = \emptyset,\qquad\forall i\in J,
\end{equation} 
for some infinite set of indexes  $J\subset\N$. Up to choosing a smaller $\varepsilon_0$,  \eqref{eq:d(a,xi)} implies that
$B_{\varepsilon_0 r_i} (x_i) \subset B_{2 C' r_i} (a)$ and so by the 
estimate \eqref{eq:genDoub} we get 
\begin{equation}\label{geq:mBe}
\mm( B_{\varepsilon_0 r_i} (x_i)) \geq  C'' \mm(B_{2 C r_i} (a)),\qquad\forall i\in \N,
\end{equation}
for some constant $C''$ independent on $i$ (but possibly depending on all other objects).
Combining \eqref{eq:AcapB}, \eqref{geq:mBe}  we get
\begin{equation} \nonumber
\mm(A\cap B_{2C r_i}(a)) \leq \mm( B_{2C r_i}(a) \setminus B_{\varepsilon_0 r_i} (x_i)) \leq (1-C'') \mm(B_{2 C r_i} (a)),\qquad\forall i\in J,
\end{equation}
and thus
\begin{equation*}
\limi_{i\to \infty} \frac{\mm(A\cap B_{2C r_i}(a))}{\mm(B_{2C r_i}(a))} \leq   1-C'' < 1
\end{equation*}
contradicting that $a$ is an $\mm$-density point of $A$, namely \eqref{eq:amDP}.

Hence for every $\eps>0$ eventually it holds $A\cap B_{\eps r_i}(x_i)\neq\emptyset$. With a diagonalization argument we can then find a sequence $(a_i)\subset A$ such that $\lim_{i\to\infty} r_i^{-1}\sfd(a_i,x_i)=0$. Recalling that  $\sfd_Z(\iota_i(a_i),\iota_i(x_i))=r_i \sfd(a_i,x_i)$, our claim \eqref{eq:claim} follows from \eqref{eq:xitoy'}.

\bigskip

By \eqref{eq:claim} and directly from the definition of p-mGH convergence, using the same space $(Z,\sfd_Z)$ and the same embeddings $\iota_i$ granting the convergence in \eqref{eq:XriatoY} we deduce that
\begin{equation*}
(X,r_i^{-1} \sfd, \mm^a_{r_i}, a_i) \to (Y,\sfd_Y,\nn, y') \quad \text{p-mGH}.
\end{equation*}
Since in the normalization \eqref{eq:normalization} we use functions of the form
 $\chi_{B_r(\bar x)}(\cdot) \; (1-\sfd_Z(\cdot, \bar x)) \in C_b$, from weak convergence it follows that
\begin{equation*}
(X,r_i^{-1} \sfd, \mm^{a_i}_{r_i}, a_i) \to (Y,\sfd_Y,\nn_1^{y'}, y') \quad \text{p-mGH},
\end{equation*}
and thus for $i$ large enough  
\begin{equation}\label{eq:est3}
\mathcal D_{C(\cdot)}\big((X,r_i^{-1} \sfd, \mm^{a_i}_{r_i}, a_i), (Y,\sfd_Y,\nn^{y'}_{1}, y')\big) \leq \frac{1}{2k}.
\end{equation} 

Since by construction we have $a_i \in A$, there exist spaces $(Y_i,\sfd_i,\nn_i,y_i) \in \Tan(X,\sfd,\mm,a_i)$ and points $y'_i \in Y_i$
such that $(Y_i,\sfd_i,(\nn_i)_{1}^{y'_i},y'_i) \in \mathcal U$ and
\begin{equation}\label{eq:y'i}
 \mathcal D_{C(\cdot)}\left((Y_i,\sfd_i,(\nn_i)_{1}^{y'_i},y'_i),(X,r^{-1}\sfd,\mm^{a_i}_r, a_i)\right) \geq \frac 2 k\quad \text{ for all } r \in (0,m^{-1}),
\end{equation}
where $(\nn_i)_{1}^{y'_i}$ is the normalization of the measure $\nn_i$ at $y_i'$ as in \eqref{eq:normalization}.

Therefore by combining the bound $\diam_{\mathcal D_{C(\cdot)}}(\mathcal U)\leq (2k)^{-1}$ with \eqref{eq:est3} and  \eqref{eq:y'i},  for sufficiently large $i$  we have
\begin{align*}
 \frac2k & \leq  \mathcal D_{C(\cdot)}\left((Y_i,\sfd_i,(\nn_i)_1^{y'_i},y'_i),(X,r_i^{-1}\sfd,\mm^{a_i}_{r_i},a_i)\right)\\
& \leq \mathcal D_{C(\cdot)}\left((Y_i,\sfd_i,(\nn_i)_{1}^{y'_i},y'_i),(Y, \sfd_Y,\nn^{y'}_{1},y')\right) + \mathcal D_{C(\cdot)}\left((Y, \sfd_Y,\nn^{y'}_{1},y'),(X,r_i^{-1}\sfd,\mm^{a_i}_{r_i},a_i)\right) \\
& \leq  \diam_{\mathcal D_{C(\cdot)}}(\mathcal U) + \mathcal D_{C(\cdot)}\left((Y, \sfd_Y,\nn^{y'}_{1},y'),(X,r_i^{-1}\sfd,\mm^{a_i}_{r_i},a_i)\right) \le \frac1{2k} + \frac1{2k} = \frac1k,
\end{align*}
which is a contradiction. 
\end{proof}

\subsection{Iterating tangents to conclude}
\emph{Proof of Theorem \ref{thm:euclideantangents}}
Let $Z \subset X$ be the set of full $\mm$-measure where both Theorem \ref{thm:iteratedtangents} and Lemma \ref{lma:prolonging} hold, and fix $x \in Z$. We will prove that there exists a tangent space to $x$ isomorphic to $(\R^n,\sfd_E,\LL_n,0)$ for some $n\leq N$.

Thanks to  Lemma \ref{lma:prolonging} there exists a non-constant geodesic $\gamma \in \geo(X)$ so that $\gamma_\frac{1}{2} = x$, therefore every tangent $(Y_1,\sfd_1,\nn_1,y_1) \in \Tan(X,\sfd, \mm,x)$ contains an isometric image of $\R$ going through the point $y_1$. 

As a tangent of an infinitesimally Hilbertian $\CD^*(K,N)$-space, $(Y_1,\sfd_1,\nn_1)$ is an infinitesimally Hilbertian $\CD^*(0,N)$-space (property \eqref{eq:tancd})
and so by Theorem \ref{thm:splitting} it splits into $(\R \times X_1, \sfd_E\times \sfd'_1,\LL_1\times \mm_1)$ with $(X_1,\sfd'_1,\mm_1)$ infinitesimally Hilbertian $\CD^*(0,N-1)$-space.

If $X_1$ is not a singleton, it contains a point $x_1$ where again both Theorem \ref{thm:iteratedtangents} and Lemma \ref{lma:prolonging}
can be used. Therefore, again by Theorem \ref{thm:splitting}, every tangent space $(Y_2,\sfd_2,\nn_2,y_2) \in \Tan(X_1,\sfd_1,\mm_1,x_1)$
 splits as $(\R \times X_2, \sfd_E\times \sfd'_2,\LL_1\times \mm_2)$ with $(X_2,\sfd'_2,\mm_2)$ infinitesimally Hilbertian $\CD^*(0,N-2)$-space.

By Theorem  \ref{thm:splitting}  this process can be iterated at most $\text{`integer part of $N$'}$-times before producing a space $X_n$. Tracing back the lines that have been factorized we conclude that  $(\R^n,\sfd_E,\LL_n,0) \in \Tan(X,\sfd,\mm,x)$ for some $n\leq N$, as desired.
{\penalty-20\null\hfill$\square$\par\medbreak}

\def\cprime{$'$}

\end{document}